\documentclass[12pt]{article}
\usepackage{amssymb,latexsym,amsmath,enumerate,verbatim,amsfonts,amsthm}
\usepackage[active]{srcltx}

\numberwithin{equation}{section}
\newtheorem{Theorem}{Theorem}[section]
\newtheorem{Definition}{Definition}[section]
\newtheorem{Proposition}[Theorem]{Proposition}

\newtheorem{Lemma}{Lemma}[section]
\newtheorem{Corollary}[Theorem]{Corollary}
\newtheorem{Remark}{Remark}[section]
\newtheorem{Example}{Example}[section]

\newtheorem{Question}{Question}[section]

\makeatletter
 \def\@biblabel#1{#1.}
\makeatother

\newcommand{\uu}{{\bf u}}

\newcommand{\x}{{\bf x}}
\newcommand{\y}{{\bf y}}
\newcommand{\z}{{\bf z}}
\newcommand{\q}{{\bf q}}

\newcommand{\e}{{\bf e}}
\newcommand{\0}{{\bf 0}}

\newcommand{\w}{{\bf w}}

\begin{document}
\title{Properties of Tensor Complementarity Problem and Some Classes of Structured Tensors}

\author{Yisheng Song\thanks{ Corresponding author. School of Mathematics and Information Science, Henan Normal University, XinXiang HeNan,  P.R. China, 453007.
Email: songyisheng1@gmail.com. This author's work was partially supported by
the National Natural Science Foundation of P.R. China (Grant No.
11571095, 11601134).} \quad Liqun Qi\thanks{ Department of Applied Mathematics, The Hong Kong Polytechnic University, Hung Hom,
Kowloon, Hong Kong. Email: maqilq@polyu.edu.hk. This author's work was supported by the Hong
Kong Research Grant Council (Grant No. PolyU
502111, 501212, 501913 and 15302114).}}

\date{}

 \maketitle

%---------------------------------------------------------------------------------Abstract
\begin{abstract}
\noindent  %\vspace{3mm}
  This paper deals with the class of Q-tensors, that is,  a Q-tensor is a real tensor $\mathcal{A}$ such that  the tensor complementarity problem $(\q, \mathcal{A})$:
	$$\mbox{ finding } \x \in \mathbb{R}^n\mbox{ such that }\x \geq \0, \q + \mathcal{A}\x^{m-1} \geq \0, \mbox{ and }\x^\top (\q + \mathcal{A}\x^{m-1}) = 0, $$ has a  solution for each vector $\q \in \mathbb{R}^n$. Several subclasses of Q-tensors are given: P-tensors, R-tensors, strictly semi-positive tensors and semi-positive R$_0$-tensors.  We prove that a nonnegative tensor is a Q-tensor if and only if all of its principal  diagonal entries are positive, and so the equivalence of Q-tensor, R-tensors, strictly semi-positive tensors is showed if they are nonnegative tensors. We also show that a tensor is a R$_0$-tensor if and only if the tensor complementarity problem $(\0, \mathcal{A})$ has no non-zero vector solution, and a tensor is a R-tensor if and only if it is a R$_0$-tensor and the tensor complementarity problem $(\e, \mathcal{A})$ has no non-zero vector solution, where $\e=(1,1\cdots,1)^\top$.

\noindent {\bf Key words:}\hspace{2mm} Q-tensor, R-tensor, R$_0$-tensor,  strictly semi-positive, tensor complementarity problem.  \vspace{3mm}

\noindent {\bf AMS subject classifications (2010):}\hspace{2mm} 65H17, 15A18, 90C30
47H15, 47H12, 34B10, 47A52, 47J10, 47H09, 15A48, 47H07.
  \vspace{3mm}

\end{abstract}

%------------------------------------------------------------------------------Section 1

\section{Introduction}
\hspace{4mm}
Throughout this paper, we use small letters $x, u, v, \alpha, \cdots$, for scalars, small bold
letters $\x, \y, \uu, \cdots$, for vectors, capital letters $A, B,
\cdots$, for matrices, calligraphic letters $\mathcal{A}, \mathcal{B}, \cdots$, for
tensors. All the tensors discussed in this paper are real.  Let $I_n := \{ 1,2, \cdots, n \}$, and $\mathbb{R}^n:=\{(x_1, x_2,\cdots, x_n)^\top;x_i\in  \mathbb{R},  i\in I_n\}$, $\mathbb{R}^n_{+}:=\{x\in \mathbb{R}^n;x\geq\0\}$, $\mathbb{R}^n_{-}:=\{\x\in \mathbb{R}^n;x\leq\0\}$, $\mathbb{R}^n_{++}:=\{\x\in \mathbb{R}^n;x>\0\}$, $\e=(1,1,\cdots,1)^\top$, and $\x^{[m]} = (x_1^m, x_2^m,\cdots, x_n^m)^\top$ for $\x = (x_1, x_2,\cdots, x_n)^\top$, where $\mathbb{R}$ is the set of real numbers, $\x^\top$ is the transposition of a vector $\x$, and $\x\geq\0$ ($\x>\0$) means $x_i\geq0$ ($x_i>0$) for all $i\in I_n$.

Let $A = (a_{ij})$ be an $n \times n $ real matrix. $A$ is said to be a {\bf Q-matrix} iff the linear complementarity problem, denoted by $(\q,A)$,
\begin{equation}\label{eq:11}
\mbox{ find } \z \in \mathbb{R}^n\mbox{ such that }\z \geq \0, \q + A\z \geq \0, \mbox{ and }\z^\top (\q + A\z) = 0
\end{equation}
has a  solution for each vector $\q \in \mathbb{R}^n$. We say that
$A$ is a  {\bf P-matrix} iff for any nonzero vector $\x$ in $\mathbb{R}^n$,
there exists $i\in I_n$ such that\ $x_i (Ax)_i > 0.$
It is well-known that $A$ is a P-matrix  if and only if the linear complementarity problem $(\q,A)$ has a unique solution for all $\q \in \mathbb{R}^n$.  Xiu and Zhang \cite{XZ} also gave the necessary and sufficient conditions of P-matrices. A good review of P-matrices and Q-matrices may be found in the books by Berman and
Plemmons \cite{BP}, and Cottle, Pang and Stone \cite{CPS}.

 Q-matrices and P(P$_0$)-matrices have a long history and wide applications in mathematical
 sciences. Pang \cite{P79} showed that each semi-monotone R$_0$-matrix is a Q-matrix. Pang \cite{P81}
 gave a class of Q-matrices which includes N-matrices and strictly semi-monotone matrices. Murty \cite{M72} showed that  a nonnegative matrix is a Q-matrix if and only if its all diagonal entries are positive.
 Morris \cite{M88} presented two counterexamples of the Q-Matrix conjectures: a matrix
 is  Q-matrix solely by considering the signs of its subdeterminants.    Cottle \cite{C80} studied some properties of complete Q-matrices, a subclass of Q-matrices.  Kojima and Saigal \cite{KR79} studied the number of solutions to a class of linear complementarity
problems. Gowda \cite{S90} proved that  a symmetric semi-monotone matrix is a Q-matrix if and only if it is an R$_0$-matrix. Eaves \cite{E71} obtained the equivalent definition of strictly semi-monotone matrices, a main subclass of Q-matrices.  \\

On the other hand, motivated by the discussion on positive
definiteness of multivariate homogeneous polynomial forms \cite{BM, HH, JM},
in 2005, Qi \cite{Qi} introduced the concept of
positive (semi-)definite symmetric tensors. In the same time, Qi  also
introduced eigenvalues, H-eigenvalues, E-eigenvalues and
Z-eigenvalues for symmetric tensors.    It was shown that an even
order symmetric tensor is positive (semi-)definite if and only if
all of its H-eigenvalues or Z-eigenvalues are positive (nonnegative)
(\cite[Theorem 5]{Qi}). Various structured tensors were studied well.   For example,
Zhang, Qi and Zhou \cite{ZQZ} and
Ding, Qi and Wei \cite{DQW} for M-tensors, Song and Qi \cite{SQ-15} for P-(P$_0$)tensors and B-(B$_0$)tensors, Qi and Song \cite{QS} for positive (semi-)definition of B-(B$_0$)tensors,  Song and Qi \cite{SQ1} for infinite and finite dimensional Hilbert tensors, Song and Qi \cite{SQ 15} for  structure properties and an equivalent definition of (strictly) copositive tensors, Chen and Qi \cite{CQ} for Cauchy tensor, Song and Qi \cite{SQ} for E-eigenvalues of weakly symmetric nonnegative tensors and so on. Beside automatical control, positive
semi-definite tensors also found applications in magnetic resonance
imaging \cite{CDHS, HHNQ, QYW, QYX}
and spectral hypergraph theory \cite{HQ12, LQY, Qi14}. Recently,  Song and Qi \cite{SQ-2015} extended the linear complementarity problem  to the tensor complementarity problem, a special class of nonlinear complementarity problems, denoted by TCP$(\q,\mathcal{A})$:  finding $\x \in \mathbb{R}^n$ such that 	$$ \x \geq \0, \q + \mathcal{A}\x^{m-1} \geq \0, \mbox{ and }\x^\top (\q + \mathcal{A}\x^{m-1}) = 0\leqno{\bf TCP(\q,\mathcal{A})}$$
or showing that no such vector exists. \\

 Very recently, an $n-$person noncooperative game was converted by Huang and Qi \cite{HQ} to a tensor complementarity problem. Furthermore, they gave the equivalence between a Nash equilibrium point of the multilinear game and a solution of the tensor complementarity problem.  The equivalence between  (strictly) semi-positive tensors and (strictly) copositive tensors in the case of symmetry were showed by Song and Qi \cite{SQ15}. The existence and uniqueness of solution of TCP$(\q,\mathcal{A})$ with some special tensors were  discussed  by  Che, Qi, Wei \cite{CQW}. The boundedness of the solution set of  the TCP$(\q,\mathcal{A})$ was studied by Song and Yu \cite{SY15}. The sparsest solutions to TCP$(\q,\mathcal{A})$ with a Z-tensor and its method to calculate were obtained by  Luo, Qi and Xiu \cite{LQX}. The equivalent conditions of solution to TCP$(\q,\mathcal{A})$ were showed by Gowda, Luo, Qi and Xiu \cite{GLQX} for a Z-tensor $\mathcal{A}$.  The global uniqueness of solution of  TCP$(\q,\mathcal{A})$ was considered by  Bai, Huang and Wang \cite{BHW} for a strong P-tensor $\mathcal{A}$. The solvability of  TCP$(\q,\mathcal{A})$ was given by Wang, Huang and Bai \cite{WHB} for a class of exceptionally regular tensors $\mathcal{A}$. The properties of  TCP$(\q,\mathcal{A})$ was studied by Ding, Luo and Qi \cite{DLQ} for   a new class of P-tensor $\mathcal{A}$.  The nice properties of the several classes of  Q-tensors were presented by Suo and Wang \cite{HSW}. The properties and  algorithm of the tensor eigenvalue complementarity problem were studied by
Song and Qi \cite{SQ13}, Ling, He, Qi \cite{LHQ2015, LHQ15}, Chen, Yang, Ye \cite{CYY}, respectively. \\

The following  questions are natural. Can we extend the concept of Q-matrices to Q-tensors?  If this can be done, are those nice properties of Q-matrices still true for Q-tensors?\\

 In this paper, we will introduce the concept of Q-tensors (Q-hypermatrices) and will study some subclasses and nice properties of such tensors.\\

In Section 2, we will extend the concept of Q-matrices to Q-tensors.  Serval main subclasses of Q-matrices also are extended to the corresponding subclasses of Q-tensors: R-tensors, R$_0$-tensors, semi-positive tensors, strictly semi-positive tensors. We will give serval examples to verify that the class of R-(R$_0$-)tensors properly contains strictly semi-positive tensors as a subclass, while the class of P-tensors  is a subclass of strictly semi-positive tensors. Some basic definitions and facts also are given in this section.

In Section 3, we will study some properties of Q-tensors. Firstly,  the equivalent definition of R-tensors is given: a tensor is a R$_0$-tensor if and only if the tensor complementarity problem $(\0, \mathcal{A})$ has not non-zero vector solution and a tensor is a R-tensor if and only if it is a R$_0$-tensor and the tensor complementarity problem $(\e, \mathcal{A})$ has not non-zero vector solution, where $\e=(1,1\cdots,1)^\top$.
 Subsequently, we will prove that each R-tensor is certainly a Q-tensor and each semi-positive R$_0$-tensor is a R-tensor. Thus, we show that every P-tensor is a Q-tensor. We will show that a nonnegative tensor is a Q-tensor if and only if all of its principal  diagonal elements are positive, and so the relationship of several structured sensors are given.  It will be proved that $\0$ is the unique feasible solution of the tensor complementarity problem $(\q,\mathcal{A})$ for $\q\geq\0$ if $\mathcal{A}$ is a non-negative Q-tensor.  %give some sufficient and necessary conditions for a tensor to be a P (P$_0$) tensor.

\section{Preliminaries}

In this section, we will define the notation and collect some basic definitions and facts, which will be
used later on.

A real $m$th order $n$-dimensional tensor (hypermatrix) $\mathcal{A} = (a_{i_1\cdots i_m})$ is a multi-array of real entries $a_{i_1\cdots
	i_m}$, where $i_j \in I_n$ for $j \in I_m$. Denote the set of all
real $m$th order $n$-dimensional tensors by $T_{m, n}$. Then $T_{m,
	n}$ is a linear space of dimension $n^m$. Let $\mathcal{A} = (a_{i_1\cdots
	i_m}) \in T_{m, n}$. If the entries $a_{i_1\cdots i_m}$ are
invariant under any permutation of their indices, then $\mathcal{A}$ is
called a {\bf symmetric tensor}.
 The zero tensor in $T_{m, n}$ is denoted by $\mathcal{O}$.   Let $\mathcal{A} =(a_{i_1\cdots i_m}) \in
T_{m, n}$ and $\x \in \mathbb{R}^n$. Then $\mathcal{A} \x^{m-1}$ is a vector in $\mathbb{R}^n$ with
its $i$th component as
$$\left(\mathcal{A} \x^{m-1}\right)_i: = \sum_{i_2, \cdots, i_m=1}^n a_{ii_2\cdots
	i_m}x_{i_2}\cdots x_{i_m}$$ for $i \in I_n$.
  We now give the
definition of Q-tensors, which are natural extensions of Q-matrices.

\begin{Definition} \label{d11}
	Let $\mathcal{A}  = (a_{i_1\cdots i_m}) \in T_{m, n}$.   We say that $\mathcal{A}$ is a {\bf Q-tensor} iff the tensor complementarity problem, denoted by $(\q, \mathcal{A})$,
	\begin{equation}\label{eq:12}\mbox{ finding } \x \in \mathbb{R}^n\mbox{ such that }\x \geq \0, \q + \mathcal{A}\x^{m-1} \geq \0, \mbox{ and }\x^\top (\q + \mathcal{A}\x^{m-1}) = 0,\end{equation} has a  solution for each vector $\q \in \mathbb{R}^n$.
\end{Definition}

\begin{Definition} \label{d21}\em
Let $\mathcal{A}  = (a_{i_1\cdots i_m}) \in T_{m, n}$.   We say that $\mathcal{A}$ is\begin{itemize}
\item[(i)]  a {\bf R-tensor} iff the following system is inconsistent
\begin{equation}\label{eq:21}\begin{cases}
0\ne\x\geq0, \ t\geq0\\ \left(\mathcal{A} \x^{m-1}\right)_i+t=0 \mbox{ if } x_i>0,\\
\left(\mathcal{A} \x^{m-1}\right)_j+t\geq0 \mbox{  if } x_j=0;\end{cases}\end{equation}
\item[(ii)] a {\bf R$_0$-tensor} iff the system (\ref{eq:21}) is inconsistent for $t=0$.\end{itemize}
\end{Definition}

Clearly, Definition \ref{d21} is a natural extension of the definition of  Karamardian's class of regular matrices \cite{K72}.

\begin{Definition} \label{d22}\em
	Let $\mathcal{A}  = (a_{i_1\cdots i_m}) \in T_{m, n}$.   $\mathcal{A}$ is said to be\begin{itemize}
\item[(i)] {\bf semi-positive} iff for each $\x\geq0$ and $\x\ne\0$, there exists an index $k\in I_n$ such that $$x_k>0\mbox{ and }\left(\mathcal{A} \x^{m-1}\right)_k\geq0;$$
\item[(ii)]  {\bf strictly semi-positive} iff for each $\x\geq\0$ and $\x\ne\0$, there exists an index $k\in I_n$ such that $$x_k>0\mbox{ and }\left(\mathcal{A} \x^{m-1}\right)_k>0;$$
\item[(iii)] a {\bf P-tensor}(Song and Qi \cite{SQ-15}) iff for each $\x$ in $\mathbb{R}^n$ and $\x\ne\0$, there
	exists $i \in I_n$ such that
	$$x_i \left(\mathcal{A} \x^{m-1}\right)_i > 0;$$
\item[(iv)] a {\bf P$_0$-tensor}(Song and Qi \cite{SQ-15}) iff for every $\x$ in $\mathbb{R}^n$ and $\x\ne\0$, there
	exists $i \in I_n$ such that $x_i \not = 0$ and
	$$x_i \left(\mathcal{A} \x^{m-1}\right)_i \geq 0.$$\end{itemize}
\end{Definition}
Clearly, each P$_0$-tensor is certainly semi-positive. The concept of P-(P$_0$-)tensor is introduced by Song and Qi \cite{SQ-15}. Furthermore, Song and Qi \cite{SQ-15} studied some nice properties of such a class of tensors. The definition of (strictly) semi-positive tensor is a natural extension of the concept of  (strictly) semi-positive (or semi-monotone)  matrices \cite{E71,FP66}.

It follows from Definitions \ref{d21} and \ref{d22} that
each P-tensor must be strictly semi-positive and every strictly semi-positive tensor is certainly both R-tensor and R$_0$-tensor. Now we give several examples to demonstrate that the above inclusions are proper.
\begin{Example}\label{e21}\em
Let $\hat{\mathcal{A}}  = (a_{i_1\cdots i_m}) \in T_{m, n}$ and $a_{i_1\cdots i_m}=1$ for all $i_1,i_2,\cdots, i_m\in I_n$. Then
$$\left(\hat{\mathcal{A}} \x^{m-1}\right)_i=(x_1+x_2+\cdots+x_n)^{m-1}$$ for all $i\in I_n$ and hence $\hat{\mathcal{A}}$ is strictly semi-positive. However, $\hat{\mathcal{A}}$ is not a P-tensor (for example, $x_i\left(\hat{\mathcal{A}} \x^{m-1}\right)_i=0$ for  $\x=(1,-1,0,\cdots,0)^\top$ and all $i\in I_n$).
\end{Example}
\begin{Example}\label{e22}\em
Let $\tilde{\mathcal{A}}  = (a_{i_1i_2i_3}) \in T_{3, 2}$ and $a_{111}=1$, $a_{122}=-1$,$a_{211}=-2$, $a_{222}=1$ and all other $a_{i_1i_2i_3}=0$. Then
$$\tilde{\mathcal{A}} \x^2=\left(\begin{aligned}x_1^2&-x_2^2\\
-2x_1^2&+x_2^2\end{aligned}\right).$$ Clearly,  $\tilde{\mathcal{A}}$ is not strictly semi-positive (for example, $\left(\tilde{\mathcal{A}} \x^2\right)_1=0$ and $\left(\tilde{\mathcal{A}} \x^2\right)_2=-1$ for  $\x=(1,1)^\top$).

$\tilde{\mathcal{A}}$ is a R$_0$-tensor. In fact,  \begin{itemize}
\item[(i)] if $x_1>0$,  $\left(\tilde{\mathcal{A}} \x^2\right)_1=x_1^2-x_2^2=0$. Then $x_2^2=x_1^2$, and so $x_2>0$, but $\left(\tilde{\mathcal{A}} \x^2\right)_2=-2x_1^2+x_2^2=-x_1^2<0$;
\item[(ii)]  if $x_2>0$,  $\left(\tilde{\mathcal{A}} \x^2\right)_2=-2x_1^2+x_2^2=0$. Then $x_1^2=\frac12x_2^2>0$, but  $\left(\tilde{\mathcal{A}} \x^2\right)_1=x_1^2-x_2^2=-\frac12x_2^2<0$.\end{itemize}

$\tilde{\mathcal{A}}$ is not a R-tensor. In fact,
 if $x_1>0$,  $\left(\tilde{\mathcal{A}} \x^2\right)_1+t=x_1^2-x_2^2+t=0$. Then $x_2^2=x_1^2+t>0$, and so $x_2>0$,   $\left(\tilde{\mathcal{A}} \x^2\right)_2+t=-2x_1^2+x_2^2+t=-x_1^2+2t$. Taking $x_1=a>0$, $t=\frac12a^2$ and $x_2=\frac{\sqrt{6}}2a$. That is, $\x=a(1,\frac{\sqrt{6}}2)^\top$ and $t=\frac12a^2$ solve  the system (\ref{eq:21}).
\end{Example}
\begin{Example}\label{e23}\em
Let $\bar{\mathcal{A}}  = (a_{i_1i_2i_3}) \in T_{3, 2}$ and $a_{111}=-1$, $a_{122}=1$,$a_{211}=-2$, $a_{222}=1$ and all other $a_{i_1i_2i_3}=0$. Then
$$\bar{\mathcal{A}} \x^2=\left(\begin{aligned}-x_1^2&+x_2^2\\
-2x_1^2&+x_2^2\end{aligned}\right).$$ Clearly,  $\bar{\mathcal{A}}$ is not strictly semi-positive (for example, $\x=(1,1)^\top$).

 $\bar{\mathcal{A}}$ is a R-tensor. In fact,  \begin{itemize}
\item[(i)] if $x_1>0$,  $\left(\bar{\mathcal{A}} \x^2\right)_1+t=-x_1^2+x_2^2+t=0$. Then $x_2^2=x_1^2-t$, but $\left(\bar{\mathcal{A}} \x^2\right)_2+t=-2x_1^2+x_2^2+t=-x_1^2<0$;
\item[(ii)] if $x_2>0$,  $\left(\bar{\mathcal{A}} \x^2\right)_2+t=-2x_1^2+x_2^2+t=0$. Then $x_1^2=\frac12(x_2^2+t)>0$, but  $\left(\bar{\mathcal{A}} \x^2\right)_1+t=-x_1^2+x_2^2+t=\frac12(x_2^2+t)>0$.\end{itemize}

$\bar{\mathcal{A}}$ is a R$_0$-tensor. In fact,  \begin{itemize}
\item[(i)] if $x_1>0$,  $\left(\bar{\mathcal{A}} \x^2\right)_1=-x_1^2+x_2^2=0$. Then $x_2^2=x_1^2$, and so $x_2>0$, but $\left(\bar{\mathcal{A}} \x^2\right)_2=-2x_1^2+x_2^2=-x_1^2<0$;
\item[(ii)] if $x_2>0$,  $\left(\bar{\mathcal{A}} \x^2\right)_2=-2x_1^2+x_2^2=0$. Then $x_1^2=\frac12x_2^2>0$, but  $\left(\bar{\mathcal{A}} \x^2\right)_1=-x_1^2+x_2^2=\frac12x_2^2>0$.\end{itemize}
\end{Example}

\begin{Lemma}{\em (\cite[Corollary 3.5]{BP})}\label{le:21}\em Let $S=\{\x\in\mathbb{R}^{n+1}_+;\sum\limits_{i=1}^{n+1}x_i=1\}$. Assumed that $F:S\to \mathbb{R}^{n+1}$ is continuous on $S$. Then there exists $\bar{\x}\in S$ such that
\begin{align}
\x^\top F(\bar{\x})&\geq \bar{\x}^\top F(\bar{\x})\ \mbox{\em for all }\ \x\in S\\
\left(F(\bar{\x})\right)_k&=\min_{i\in I_{n+1}}\left(F(\bar{\x})\right)_i=\omega\ \mbox{\em if }\ x_k>0,\\
\left(F(\bar{\x})\right)_k&\geq \omega\ \mbox{\em if }\ x_k=0.
\end{align}
\end{Lemma}

Recall that a tensor $\mathcal{C} \in T_{m, r}$  is called {\bf a principal sub-tensor}  of a tensor $\mathcal{A} = (a_{i_1\cdots i_m}) \in T_{m, n}$ ($1 \le r\leq n$) iff there is a set $J$ that composed of $r$ elements in $I_n$ such that
 $$\mathcal{C} = (a_{i_1\cdots i_m}),\mbox{ for all } i_1, i_2, \cdots, i_m\in J.$$ The concept was first introduced and used in \cite{Qi} for symmetric tensor. We denote by $\mathcal{A}^J_r$ the principal sub-tensor of a tensor $\mathcal{A} \in T_{m, n}$ such that the entries of  $\mathcal{A}^J_r$ are indexed by $J \subset I_n$ with $|J|=r$ ($1 \le r\leq n$), and denote by $\x_J$ the $r$-dimensional sub-vector of a vector $\x \in \mathbb{R}^n$, with the components of $\x_J$ indexed by $J$.   Note that for $r=1$, the principal sub-tensors are just the diagonal entries.

\begin{Definition} \label{d23}\em
	Let $\mathcal{A}  = (a_{i_1\cdots i_m}) \in S_{m, n}$.  $\mathcal{A}$ is said to be \begin{itemize}
\item[(i)] {\bf copositive } if $\mathcal{A}x^m\geq0$ for all $x\in \mathbb{R}^n_+$;
\item[(ii)] {\bf strictly copositive} if  $\mathcal{A}x^m>0$ for all $x\in \mathbb{R}^n_+\setminus\{0\}$.\end{itemize}
\end{Definition}

  The concept of (strictly) copositive tensors was first introduced by Qi in \cite{Qi1}.  Song and Qi \cite{SQ15} showed their equivalent definition and some special structures. The following lemma is one of the structure conclusions  of (strictly) copositive tensors in \cite{SQ15}.

\begin{Lemma}\em { (\cite[Corollary 4.6]{SQ15})}\label{le:22}
Let $\mathcal{A} = (a_{i_1\cdots i_m}) \in S_{m, n}$. Then \begin{itemize}
\item[(i)] If $\mathcal{A}$ is  copositive, then $a_{ii\cdots i}\geq0$ for all $i\in I_n.$
\item[(ii)]If $\mathcal{A}$ is strictly copositive, then $a_{ii\cdots i}>0$ for all $i\in I_n.$
\end{itemize}
\end{Lemma}

\begin{Definition} \label{d25}\em Given a function $F : \mathbb{R}^n_+
\to \mathbb{R}^n,$  the nonlinear
complementarity problem,  denoted by NCP$(F)$, is to
\begin{equation}\label{eq:26}\mbox{ find a vector } \x \in \mathbb{R}^n\mbox{ such that }\x \geq \0, F(\x) \geq \0, \mbox{ and }\x^\top F(\x) = 0,\end{equation}\end{Definition}

It is well known that the nonlinear complementarity problems have been widely applied to the field of transportation planning, regional science, socio-economic analysis, energy modeling, and game theory. So over the past decades, the solutions of nonlinear complementarity problems have been rapidly studied in its theory of existence, uniqueness and algorithms. The following conclusion (Theorem \ref{th21}) is one of  the most fundamental results, which is showed with the help of the topological degree theory and the monotone properities of the function.

\begin{Definition} (\cite{HXQ,FP11}) \label{d26}\em   A mapping $F: K\subset\mathbb{R}^n \to \mathbb{R}^n$ is said to be
\begin{itemize}
\item[(i)] pseudo-monotone on $K$ if for all vectors $\x,\y \in K$,
$$( \x -\y )^\top F(\y) \geq 0\ \Rightarrow\ ( \x -\y )^\top F(x)\geq0;$$
\item[(ii)] monotone on $K$ if
$$( F(\x) - F(\y) )^\top ( \x - \y ) \geq 0, \forall x, y \in K;$$
\item[(iii)] strictly monotone on $K$ if
$$( F(\x) - F(\y) )^\top ( \x - \y ) >0, \forall x, y \in K \mbox{ and }x \ne y;$$
\item[(iv)] strongly monotone on $K$ if there exists a constant $c > 0$ such that
$$( F(\x)- F(\y) )^\top ( \x -\y ) \geq c \| \x- \y\|^2;$$
\item[(v)] a P$_0$ function on $K$ if for all pairs of distinct vectors $\x$ and $\y$ in $K$,
there exists $k \in I_n$ such that
$$x_k\ne y_k \mbox{ and }\left(x_k-y_k) (F(\x)-F(\y)\right)_k \geq 0;$$
\item[(vi)]  a P function on $K$ if for all pairs of distinct vectors $\x$ and $\y$ in $K$,
$$\max_{k\in I_n}( x_k -y_k) \left( F(\x) -F(\y) \right)_k> 0;$$
\item[(vii)]  a uniformly P function on $K$ if there exists a constant $c> 0$ such that
for all pairs of vectors $\x$ and $\y$ in $K$,
$$\max_{k\in I_n}( x_k-y_k) \left( F(\x)-F(\y) \right)_k\geq c\|x-y\|^2.$$
\end{itemize}
\end{Definition}

It follows from the above definition of the monotonicity and P properties that the following relations hold (see \cite{HXQ,FP11} for more details):  $$\begin{aligned}
\mbox{strongly }&\mbox{monotone }\Rightarrow \mbox{ strictly } &\mbox{monotone }\Rightarrow\mbox{ mono}&\mbox{tone }\Rightarrow\mbox{ pseudo-monotone}\\
\Downarrow& \ \ \ \ \ \ \ \ \ \ \ \ \ \ \ \ \ \ \ \ \ \ \ \ \ \ \ \ \Downarrow&\ \ \ \ \ \ \Downarrow&\\
\mbox{uniformly }&\mbox{P function }\Rightarrow \mbox{\ \ \ \ \ \ \ \ P  }&\mbox{function }\Rightarrow\mbox{ P$_0$ }&\mbox{function}
\end{aligned}$$

% or \cite[Theorem 2.4.4]{FP11}
 %\begin{Theorem}\em(\cite[Theorem 2.3.11]{HXQ} or \cite{FP11} )\label{th21} Let $F$ be a pseudo-monotone and continuous mapping from $\mathbb{R}^n_+$ into $\mathbb{R}^n$. If  NCP$(F)$  has a strictly feasible point $x^*$, i.e., $$x^*\geq0,\ F(x^*)>0,$$ then NCP$(F)$ has a solution.\end{Theorem}

Now we give an example to certify the function deduced by a R-tensor is neither pseudo-monotone nor a P$_0$ function. %However, it will be proved in next section to the corresponding tensor complementarity problem has a solution.

 \begin{Example} \label{e24}\em Let $\bar{\mathcal{A}}$ be a R-tensor defined by Example \ref{e23} and let $F(\x)=\bar{\mathcal{A}} \x^2+\q$, where $\q=(\frac12,\frac12)^\top$. Then  $F$ is neither pseudo-monotone nor P$_0$ function. In fact, $$F(\x)=\bar{\mathcal{A}} \x^2+\q=\left(\begin{aligned}-x_1^2&+x_2^2+\frac12\\
-2x_1^2&+x_2^2+\frac12\end{aligned}\right).$$ Let $\x=(1,0)^\top$ and  $\y=(0,\frac14)^\top$. Then $$\x-\y=\left(\begin{aligned}1\\
-\frac14\end{aligned}\right),\ F(\x)=\left(\begin{aligned}-\frac12\\
-\frac32\end{aligned}\right)\mbox{ and }F(\y)=\left(\begin{aligned}\frac9{16}\\
\frac9{16}\end{aligned}\right).$$ Clearly, we have $$(\x-\y)^\top F(\y)=1\times\frac9{16}-\frac14 \times \frac9{16}>0.$$  However, $$(\x-\y)^\top F(\x)=-\frac12-\frac14 \times (-\frac32)<0,$$ and hence $F$ is not pseudo-monotone.

%Similarly, $F(\x)=\bar{\mathcal{A}} \x^2$ is not pseudo-monotone also.

Take $\x=(1,1)^\top$ and  $\y=(0,\frac14)^\top$. Then $$\x-\y=\left(\begin{aligned}1\\
-\frac14\end{aligned}\right),\ F(\x)=\left(\begin{aligned}\frac12\\
-\frac12\end{aligned}\right)\mbox{ and }F(\y)=\left(\begin{aligned}\frac9{16}\\
\frac9{16}\end{aligned}\right).$$ Clearly, we have $$(x_1-y_1) \left(F(\x)-F(\y)\right)_1=1\times(\frac12-\frac9{16})<0$$  and  $$(x_2-y_2)\left(F(\x)-F(\y)\right)_2=\frac34 \times (-\frac12-\frac9{16})<0,$$ and hence $F$ is not a P$_0$ function.
\end{Example}

\begin{Remark}\em Let $\mathcal{A}\in T_{m,n}$ and $F(\x)=\mathcal{A} \x^{m-1}$.  %If $F$ is a P$_0$ function, $\mathcal{A}$ is P$_0$-tensor. In fact, let $\y=\0$ and $\x\in\mathbb{R}^n_+$ in Definition \ref{d26}(v). Then the claim is proved.
		%Similarly,
	Taking $\y=\0$ and $\x\in\mathbb{R}^n_+$ in Definition \ref{d26}(vi), we obtain that $\mathcal{A}$ is P-tensor if $F$ is a P function. So $\mathcal{A}$ must be a R-tensor if $F(\x)=\mathcal{A} \x^{m-1}$ is a P function. The Example \ref{e21} means that the inverse implication is not true.
\end{Remark}

Next we will show our  main result: each R-tensor $\mathcal{A}$ is a Q-tensor. That is, the  nonlinear complementarity problem, \begin{equation}\label{eq:37}\mbox{ find } \x \in \mathbb{R}^n\mbox{ such that }\x \geq \0, F(\x)=\mathcal{A} \x^{m-1}+\q \geq \0, \mbox{ and }\x^\top F(\x) = 0,\end{equation} has a  solution for each vector $\q \in \mathbb{R}^n$.

\section{Tensor Complementarity Problem and Some Classes of Structured Tensors}

\hspace{4mm}

We first give the equivent definition of  R$_0$-tensor (R-tensor) by means of the tensor complementarity problem.

\begin{Proposition} \label{th:31}\em Let $\mathcal{A}  = (a_{i_1\cdots i_m}) \in T_{m, n}$. Then
	\begin{itemize}
		\item[(i)] $\mathcal{A}$ is a R$_0$-tensor if and only if the tensor complementarity problem $(\0, \mathcal{A})$ has a unique solution $\0$;
		\item[(ii)] $\mathcal{A}$ is a R-tensor if and only if it is a R$_0$-tensor and the tensor complementarity problem $(\e, \mathcal{A})$ has a unique solution $\0$, where $\e=(1,1\cdots,1)^\top$.
		\end{itemize}	
\end{Proposition}

\begin{proof}
	(i)  The tensor complementarity problem $(\0, \mathcal{A})$ have not non-zero vector solution if and only if the system $$\begin{cases}
		0\ne\x=(x_1,\cdots,x_n)^\top\geq0,  \\
		\left(\mathcal{A} \x^{m-1}\right)_i=0 \mbox{ if } x_i>0,\\
		\left(\mathcal{A} \x^{m-1}\right)_i\geq0 \mbox{  if } x_i=0\end{cases}$$
has no solution. So the conclusion is proved.

(ii)  It follows from the Definition \ref{d21} that the necessity is obvious ($t=1$).

Conversely， suppose $\mathcal{A}$ is not a R-tensor. Then there exists $\x\in\mathbb{R}^n_+\setminus \{\0\}$ satisfying  the system (\ref{eq:21}). That is, the tensor complementarity problem $(t\e, \mathcal{A})$ has non-zero vector solution $\x$ for some $t\geq0$. We have $t>0$ since $\mathcal{A}$ is a R$_0$-tensor. So the tensor complementarity problem $(\e, \mathcal{A})$ has non-zero vector solution $\dfrac{\x}{\sqrt[m-1]{t}}$, a contradiction.
\end{proof}

Now we give the following result which may be obtained by Proposition \ref{th:31} together with the main results of Karamardian \cite{K76}. For complete, we give another proof using the similar proof technique in Berman and Plemmons \cite[Theorem 3.6]{BP}.

\begin{Corollary}\label{co:31}\em Let $\mathcal{A}  = (a_{i_1\cdots i_m}) \in T_{m, n}$ be a R-tensor.
Then $\mathcal{A}$ is a Q-tensor. That is, the tensor complementarity problem $(\q, \mathcal{A})$ has a  solution for all $\q \in \mathbb{R}^n$.
 \end{Corollary}

 \begin{proof}
 Let the mapping $F:\mathbb{R}^{n+1}_+\to \mathbb{R}^{n+1}$ be defined by
 \begin{equation}\label{eq:31}
 F(\y)=\left(\begin{aligned}
 \mathcal{A}\x^{m-1}&+s\q+s\e\\
 &s
 \end{aligned}\right),
 \end{equation}
 where $\y=(\x,s)^\top$, $\x\in\mathbb{R}^n_+$, $s\in\mathbb{R}_+$ and $\e=(1,1,\cdots,1)^\top\in\mathbb{R}^n,\ \q\in\mathbb{R}^n$. Obviously, $F:S\to \mathbb{R}^{n+1}$  is continuous on the set $S=\{\x\in\mathbb{R}^{n+1}_+;\sum\limits_{i=1}^{n+1}x_i=1\}$. It follows from Lemma \ref{le:21} that there exists $\tilde{\y}=(\tilde{\x},\tilde{s})^\top\in S$ such that \begin{align}
\y^\top F(\tilde{\y})&\geq \tilde{\y}^\top F(\tilde{\y})\ \mbox{ for all }\ \y\in S\label{eq:32}\\
\left(F(\tilde{\y})\right)_k&=\min_{i\in I_{n+1}}\left(F(\tilde{\y})\right)_i=\omega\ \mbox{ if }\ \tilde{y}_k>0,\label{eq:33}\\
\left(F(\tilde{\y})\right)_k&\geq \omega\ \mbox{ if }\ \tilde{y}_k=0.\label{eq:34}
\end{align}

We claim $\tilde{s}>0$. Suppose $\tilde{s}=0$.   Then the fact that $\tilde{y}_{n+1}=\tilde{s}=0$ together with (\ref{eq:34}) implies that $$\omega\leq\left(F(\tilde{\y})\right)_{n+1}=\tilde{s}=0,$$   and so for $k\in I_n$,
$$\begin{aligned}
\left(F(\tilde{\y})\right)_k&=\left(\mathcal{A}\tilde{\x}^{m-1}\right)_k=\omega\ \mbox{ if }\ \tilde{x}_k>0,\\
\left(F(\tilde{\y})\right)_k&=\left(\mathcal{A}\tilde{\x}^{m-1}\right)_k\geq \omega\ \mbox{ if }\ \tilde{x}_k=0.
\end{aligned}$$
That is, for $t=-\omega\geq0$,
$$\begin{aligned}
\left(\mathcal{A}\tilde{\x}^{m-1}\right)_k+t=0\ \mbox{ if }\ \tilde{x}_k>0,\\
\left(\mathcal{A}\tilde{\x}^{m-1}\right)_k+t\geq 0\ \mbox{ if }\ \tilde{x}_k=0.
\end{aligned}$$
This obtains a contradiction with the definition of R-tensor $\mathcal{A}$, which completes the proof of the claim.

 Now we show that the tensor complementarity problem $(\q, \mathcal{A})$ has a solution for all $\q \in \mathbb{R}^n$. In fact, if $\q\geq\0$, clearly $\z=\0$ and $\w=\mathcal{A}\z^{m-1}+\q=\q$ solve $(\q, \mathcal{A})$. Next we consider $\q\in \mathbb{R}^n/\mathbb{R}^n_+$. It follows from (\ref{eq:31}), (\ref{eq:33}) and (\ref{eq:34}) that  $$\left(F(\tilde{\y})\right)_{n+1}=\min_{i\in I_{n+1}}\left(F(\tilde{\y})\right)_i=\omega=\tilde{s}=\tilde{y}_{n+1}>0$$ and for $i\in I_n$,
  $$\begin{aligned}
\left(F(\tilde{\y})\right)_i&=\left(\mathcal{A}\tilde{\x}^{m-1}\right)_i+\tilde{s}q_i+\tilde{s}=\omega=\tilde{s}\ \mbox{ if }\ \tilde{y}_i=\tilde{x}_i>0,\\
\left(F(\tilde{\y})\right)_i&=\left(\mathcal{A}\tilde{\x}^{m-1}\right)_i+\tilde{s}q_i+\tilde{s}\geq \omega= \tilde{s}\ \mbox{ if }\ \tilde{y}_i=\tilde{x}_i=0.
\end{aligned}$$
Thus  for $\z=\frac{\tilde{\x}}{\tilde{s}^{\frac1{m-1}}}$ and $i\in I_n$, we have $$\begin{aligned}
\left(\mathcal{A}\z^{m-1}\right)_i+q_i=0\ \mbox{ if }\ z_i>0,\\
\left(\mathcal{A}\z^{m-1}\right)_i+q_i\geq 0\ \mbox{ if }\ z_i=0,
\end{aligned}$$
  and hence, $$\z \geq \0, \w=\q + \mathcal{A}\z^{m-1} \geq \0, \mbox{ and }\z^\top \w = 0.$$
  So we obtain a feasible solution $(\z,\w)$ of the tensor complementarity problem $(\q, \mathcal{A})$, and then $\mathcal{A}$ is a Q-tensor.
 The theorem is proved.
 \end{proof}
\begin{Corollary}\label{co:32}\em Each strictly semi-positive tensor is a Q-tensor, and so is P-tensor. That is, the tensor complementarity problem $(\q, \mathcal{A})$  has a  solution for all $\q \in \mathbb{R}^n$ if $\mathcal{A}$ is either a P-tensor or a strictly semi-positive tensor.
 \end{Corollary}

\begin{Theorem}\label{th:33}\em Let a R$_0$-tensor $\mathcal{A} (\in T_{m, n})$ be semi-positive.
Then $\mathcal{A}$ is a R-tensor, and hence $\mathcal{A}$ is a Q-tensor.
 \end{Theorem}
 \begin{proof}
 Suppose that $\mathcal{A}$ is not a R-tensor. Let the system (\ref{eq:21}) have a solution $\bar{\x}\geq0$ and $\bar{\x}\ne 0$. If $t=0$, this contradicts the assumption that  $\mathcal{A}$ is a R$_0$-tensor. So we must have $t>0$.  Then for $i\in I_n$, we have $$\left(\mathcal{A}\x^{m-1}\right)_i+t=0\ \mbox{ if }\ x_i>0,$$ and hence,$$\left(\mathcal{A}\x^{m-1}\right)_i=-t<0\ \mbox{ if }\ x_i>0,$$ which  contradicts  the assumption that  $\mathcal{A}$ is semi-positive.  So $\mathcal{A}$ is a R-tensor, and hence $\mathcal{A}$ is a Q-tensor by Corollary \ref{co:31}.
 \end{proof}
 So, the following relationship of several classes of  structured sensors hold.\\

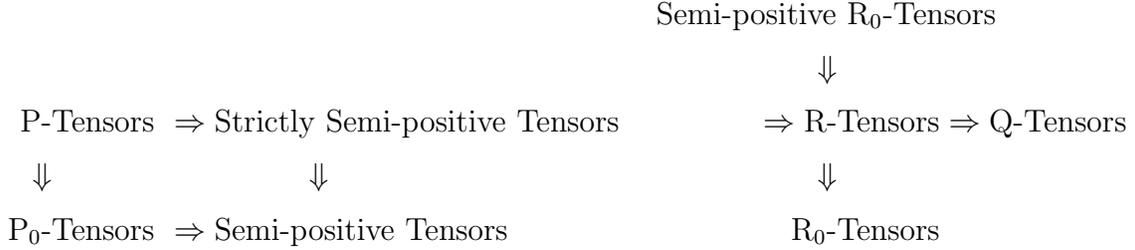
 	$$\begin{aligned}
 	\	& \ &\ \mbox{Semi-positive}&\mbox{ R$_0$-Tensors}\\
 	& \ &\Downarrow&\\
 	\mbox{P-}&\mbox{Tensors }\Rightarrow \mbox{Strictly Semi-positive  Tensors} &\Rightarrow\mbox{R-}&\mbox{Tensors}\Rightarrow\mbox{Q-Tensors}\\
 	\Downarrow& \ \ \ \ \ \ \ \ \ \ \ \ \ \ \ \ \ \ \ \ \ \ \ \ \Downarrow& \Downarrow&\\
 	\mbox{P$_0$-}&\mbox{Tensors }\Rightarrow \mbox{Semi-positive  Tensors}& \ \ \ \   \mbox{ R$_0$-}&\mbox{Tensors }
 	\end{aligned}$$\\

 \begin{Theorem}\label{th:34}\em Let  $\mathcal{A}=(a_{i_1\cdots i_m})\in T_{m, n}$ with $\mathcal{A}\geq\mathcal{O}$ ($a_{i_1\cdots i_m}\geq0$ for all $i_1\cdots i_m\in I_n$).
Then  $\mathcal{A}$ is a Q-tensor if and only if $a_{ii\cdots i}>0$ for all $i\in I_n$.
 \end{Theorem}

 \begin{proof} Sufficiency.  If $a_{ii\cdots i}>0$ for all $i\in I_n$ and $\mathcal{A}\geq\mathcal{O}$, then it follows from the definition \ref{d22} of the strictly semi-positive tensor that $\mathcal{A}$ is strictly semi-positive, and hence $\mathcal{A}$ is a Q-tensor by Corollary \ref{co:32}.

 Necessity. Suppose that there exists $k\in I_n$ such that $a_{kk\cdots k}=0$. Let $\q=(q_1,\cdots,q_n)^\top$ with $q_k<0$ and $q_i>0$ for all $i\in I_n$ and $i\ne k$. Since $\mathcal{A}$ is a Q-tensor, the tensor complementarity problem $(\q,\mathcal{A})$ has at least a solution. Let $\z$  be a feasible solution to $(\q,\mathcal{A})$.
Then \begin{equation}\label{eq:35}\z\geq\0, \ \w=\mathcal{A}\z^{m-1}+\q\geq\0\mbox{ and }\z^\top\w=0.\end{equation}
Clearly, $\z\ne\0$.  Since $\z\geq\0$ and $\mathcal{A}\geq0$ together with $q_i>0$ for each $i\in I_n$ with $i\ne k$, we must have $$ w_i=\left(\mathcal{A}\z^{m-1}\right)_i+q_i=\sum_{i_2,\cdots, i_m=1}^n a_{ii_2\cdots i_m}z_{i_2}\cdots
z_{i_m}+q_i>0 \mbox{ for } i\ne k\mbox{ and } i\in I_n .$$ It follows from (\ref{eq:35}) that $$z_i=0 \mbox{ for } i\ne k\mbox{ and } i\in I_n .$$
Thus, we have $$ w_k=\left(\mathcal{A}\z^{m-1}\right)_k+q_k=\sum_{i_2,\cdots, i_m=1}^n a_{ki_2\cdots i_m}z_{i_2}\cdots
z_{i_m}+q_k=a_{kk\cdots k}z_k^{m-1}+q_k=q_k<0$$ since $a_{kk\cdots k}=0.$ This contradicts the fact that $\w\geq\0$, so $a_{ii\cdots i}>0$ for all $i\in I_n$.
 \end{proof}
 \begin{Corollary}\label{co:35}\em Let a non-negative tensor $\mathcal{A}$  be a Q-tensor. Then all principal sub-tensors of $\mathcal{A}$ are also Q-tensors.
 \end{Corollary}
 \begin{Corollary}\label{co:36}\em Let a non-negative tensor $\mathcal{A}$  be a Q-tensor. Then  $\0$ is the unique feasible solution to the tensor complementarity problem $(\q,\mathcal{A})$ for $\q\geq\0$.
 \end{Corollary}

 \begin{proof} It follows from Theorem \ref{th:34} that $a_{ii\cdots i}>0$ for all $i\in I_n$, and hence $$\left(\mathcal{A}\x^{m-1}\right)_i=\sum_{i_2,\cdots, i_m=1}^n a_{ii_2\cdots i_m}x_{i_1}\cdots
x_{i_m}= a_{ii\cdots i}x_i^{m-1}+\sum_{(i_2,\cdots, i_m)\ne (i,\cdots,i)} a_{ii_2\cdots i_m}x_{i_1}\cdots
x_{i_m}.$$

If $\x=(x_1,\cdots,x_n)^\top$ is any feasible solution of the tensor complementarity problem $(\q,\mathcal{A})$, then we have \begin{equation}\label{eq:36}\x\geq\0, \ \w=\mathcal{A}\x^{m-1}+\q\geq\0\mbox{ and }\x^\top\w=\mathcal{A}\x^m+\x^\top\q=0.\end{equation}
 Suppose $x_i>0$ for some $i\in I_n$. Then $$w_i=\left(\mathcal{A}\x^{m-1}\right)_i+q_i=a_{ii\cdots i}x_i^{m-1}+\sum_{(i_2,\cdots, i_m)\ne (i,\cdots,i)} a_{ii_2\cdots i_m}x_{i_1}\cdots x_{i_m}+q_i>0,$$
 and hence, $\x^\top\w=x_iw_i+\sum\limits_{k\ne i}x_kw_k>0$. This contradicts the fact that $\x^\top\w=0$. Consequently, $x_i=0$ for all  $i\in I_n$.
\end{proof}

Following the above conclusions together with Theorems 3.2 and 3.4 of Song and Qi \cite{SQ15}, the following results are obvious.

 \begin{Corollary}\label{co:37}\em Let  $\mathcal{A}$  be a non-negative tensor. Then the following are equivalent:
 \begin{itemize}
 \item[(i)] $\mathcal{A}$ is a Q-tensor;
 \item[(ii)] $\mathcal{A}$ is a R-tensor;
 \item[(iii)] $\mathcal{A}$ is a  strictly semi-positive tensor;
 \item[(iv)] $a_{ii\cdots i}>0\mbox{ for all }i\in I_n$.	
 \end{itemize}	
	\end{Corollary}

  \begin{Corollary}\label{co:37}\em Let  $\mathcal{A}$  be a symmetric and non-negative tensor. Then the following are equivalent:
  	\begin{itemize}
  		\item[(i)] $\mathcal{A}$ is a Q-tensor;
  		\item[(ii)] $\mathcal{A}$ is a R-tensor;
  		\item[(iii)] $\mathcal{A}$ is a  strictly semi-positive tensor;
  		\item[(iv)] $\mathcal{A}$ is a  strictly copositive tensor;
  		\item[(v)] $a_{ii\cdots i}>0\mbox{ for all }i\in I_n$.	
  	\end{itemize}	
  \end{Corollary}

\begin{Question} \label{Q1}\em Let $\mathcal{A}$ be a Q-tensor.
 \begin{itemize}
\item Whether or not  a nonzero solution $\x$ of Tensor Complementarity Problem $(\0,\mathcal{A})$ contains at least two nonzero components if $\mathcal{A}$ is a semi-positive Q-tensor;

\item Whether or not there are some relation between the eigenvalue of (symmetric) Q-tensor and the feasible solution of Tensor Complementarity Problem  $(\q, \mathcal{A})$.\end{itemize}
\end{Question}

%%%%%%%%%%%%%%%%%%%%%%%%%%%%%%%%%%%%%%%%%%%%%%%%%%%%%%%%%%%%%%%%%%%%%%%%%%%%%%%%%%%%%%%%%%%%%%%%%%%%%%%%%%%%%%%%%%%%%%%%%%%%%%%%%%%%%%%%%%%%%%%%%%%%%%%%%%%%

%\section*{\bf Acknowledgment} The authors would like to thank  Editor, the anonymous referees and Professor Jong-Shi Pang  for their valuable suggestions which helped us to improve this manuscript.

%\bibliographystyle{amsplain}

\end{document}